\documentclass[12pt]{amsart}
\usepackage{amsmath}
\usepackage{latexsym}
\usepackage{amssymb}
\usepackage{amscd}
\usepackage{amsfonts}
\usepackage[all]{xy}

\usepackage{tabularx}
\newcolumntype{Y}{>{\small\raggedright\arraybackslash}X}

\theoremstyle{definition}

\newtheorem{thm}{Theorem}[section]

\newtheorem{question}[thm]{Question}
\newtheorem{defi}[thm]{Definition}
\newtheorem{prop}[thm]{Proposition}
\newtheorem{cor}[thm]{Corollary}
\newtheorem{lem}[thm]{Lemma}
\newtheorem{rmk}[thm]{Remark}

\newtheorem{nota}[thm]{Notation}

\newtheorem*{notations}{Notations}
\newtheorem*{acknowledgement}{Acknowledgement}

\theoremstyle{remark}

\begin{document}

\title{Derived categories of surfaces isogenous to a higher product}
\author{Kyoung-Seog Lee}
\maketitle

\begin{abstract}

Let $S=(C\times D)/G$ be a surface isogenous to a higher product of unmixed type with $p_g=q=0$, $G=(\mathbb{Z}/3)^2$. We construct exceptional sequences of line bundles of maximal length on $S$. As a consequence we find new examples of quasiphantom categories.

\end{abstract}

\section{Introduction}

Recently derived categories of surfaces of general type draw lots of
attention. Several interesting semiorthogonal decompositions of the
derived categories were constructed by B\"{o}hning, Graf von
Bothmer, and Sosna on the classical Godeaux surface {\cite{BBS}};
Alexeev and Orlov on the primary Burniat surfaces {\cite{AO}};
Galkin and Shinder on the Beauville surface {\cite{GS}};
B\"{o}hning, Graf von Bothmer, Katzarkov and Sosna on the
determinantal Barlow surfaces {\cite{BBKS}}; Fakhruddin on some fake
projective planes{\cite{F}}; Galkin, Katzarkov, Mellit and Shinder
on some different fake projective planes and on a fake cubic surface
\cite{GKMS}; Coughlan on some surfaces obtained as abelian
coverings of del Pezzo surfaces \cite{C}. These semiorthogonal
decompositions consist of admissible subcategories generated by
exceptional sequences of line bundles of maximal lengths and their
orthogonal complements. These orthogonal complements have vanishing Hochschild
homologies and finite Grothendieck groups. An admissible
triangulated subcategory of a derived category of a smooth
projective variety is called a quasiphantom category if its Hochschild
homology vanishes and its Grothendieck group is finite. When the
Grothendieck group of a quasiphantom category also vanishes it is
called a phantom category. Gorchinskiy and Orlov in {\cite{GO}}
constructed phantom categories using quasiphantom categories
constructed in {\cite{AO}}, {\cite{BBS}}, {\cite{GS}}. Determinantal
Barlow surfaces also provide examples of phantom categories
{\cite{BBKS}}.

Let $S$ be a surface isogenous to a higher product $(C \times D) /G$ of unmixed type with $p_g=q=0$. If $G$ is an abelian group, Bauer and Catanese {\cite{BC}} proved that $G$ is one of $(\mathbb{Z}/2)^3, (\mathbb{Z}/2)^4, (\mathbb{Z}/3)^2, (\mathbb{Z}/5)^2$. As mentioned above, Galkin and Shinder {\cite{GS}} constructed exceptional sequences of line bundles of maximal length on the Beauville surface which is a surface isogenous to a higher product with $p_g=q=0$ and $G=(\mathbb{Z}/5)^2$. Motivated by their work, we study the derived categories of $2$-dimensional family of surfaces isogenous to a higher product with $p_g=q=0$, $G=(\mathbb{Z}/3)^2$. In this paper, we construct exceptional sequences of line bundles of maximal length on these surfaces and prove that the complements of the admissible subcategories generated by these line bundles are quasiphantom categories. This gives new examples of quasiphantom categories having Grothendieck groups $(\mathbb{Z}/3)^5$ and these categories can be used to construct phantom categories by a theorem of Gorchinskiy and Orlov {\cite{GO}}.

The constructions of exceptional sequences are as follows. Let $S = (C \times D) /G$ be a surface isogenous to a higher product of unmixed type with $p_g=q=0$, $G=(\mathbb{Z}/3)^2$. First we show that there are exactly two $G$-invariant effective divisors of degree $3$ on $C$ which are not linearly equivalent. Let us denote them by $E_1$ and $E_2$. Then we show that $H^k(C,\mathcal{O}_C(2E_1-E_2))$ vanishes for all $k \in \mathbb{Z}$. Similary there are exactly two $G$-invariant effective divisors of degree $3$ on $D$ which are not linearly equivalent. Let us denote them by $F_1$ and $F_2$. By the same argument, $H^k(D,\mathcal{O}_D(2F_1-F_2))$ vanishes for all $k \in \mathbb{Z}$. Let $X$ be the product of $C$ and $D$. By abuse of notation, we let $\mathcal{O}_X(E_i)$ (respectively, $\mathcal{O}_X(F_i)$) for $i \in \{1,2\}$ denote the pullback of $ \mathcal{O}_C(E_i) $ (respectively, $ \mathcal{O}_D(F_i)$). For any character $\chi \in Hom(G,\mathbb{C}^*)$, we can identify the equivariant line bundles $\mathcal{O}_X(E_i)(\chi)$ (respectively, $ \mathcal{O}_X(F_i)(\chi) $) on $X$ with line bundles on $S$. Then for any choice of four characters $\chi_1, \chi_2, \chi_3, \chi_4 \in Hom(G,\mathbb{C}^*)$, we get the following sequence of line bundles on $S$. $$\mathcal{O}_X(\chi_1), \mathcal{O}_X(E_2-2E_1)(\chi_2), \mathcal{O}_X(F_2-2F_1)(\chi_3), \mathcal{O}_X(E_2-2E_1+F_2-2F_1)(\chi_4).$$
By the K\"{u}nneth formula, we find that the above sequence is an exceptional sequence. Since the rank of the Grothendieck group of $S$ is 4, we see that the above sequence is of maximal length.

We also compute Hochschild cohomologies of quasiphantom categories
and prove that for some exceptional sequences we obtained the
categories generated by those exceptional sequences are defomation
invariant. While adding these results to this paper which was on the
arXiv, similar results have been obtained independently by Coughlan in
\cite{C} via different method. In his paper \cite{C}, Coughlan considers
general type surfaces which are obtained as abelian covers of del
Pezzo surfaces satisfying some conditions. His method can be applied to surfaces isogenous to a higher product with $G=(\mathbb{Z}/3)^2$, $G=(\mathbb{Z}/5)^2$ and many other general type surfaces. He constructs many
exceptional sequences of maximal lengths on these surfaces and studies deformation invariance and Hochschild cohomologies.

This paper is organized as follows. In $\S$2, we collect some basic
facts about the surfaces isogenous to a higher product and compute
the Grothendieck groups of these surfaces. In $\S$3, we construct
exceptional sequences of line bundles on the $2$-dimensional family
of surfaces isogenous to a higher product with $p_g=q=0$,
$G=(\mathbb{Z}/3)^2$. In $\S$4, we discuss quasiphantom and phantom
categories. In $\S$5, we consider $G=(\mathbb{Z}/2)^3$,
$G=(\mathbb{Z}/2)^4$ cases.

\begin{acknowledgement}
I am grateful to my advisor Young-Hoon Kiem for his invaluable
advice and many suggestions for the first draft of this paper.
Without his support and encouragement, this work could not have been
accomplished. I thank Fabrizio Catanese for answering my questions and sending me corrections of \cite{BC}. I would like to thank Seoul National University for its support during the preparation of this paper.
\end{acknowledgement}

\begin{notations}
We will work over $\mathbb{C}$. A curve will mean a smooth projective curve. A surface will mean a smooth projective surface. Derived category of a variety will mean the bounded derived category of coherent sheaves on that variety. $G$ denotes a finite group and $\widehat{G}=Hom(G,\mathbb{C}^*)$ denotes the character group of $G$. Here $ \sim $ denotes linear equivalence of divisors.
\end{notations}

\section{Preliminaries}

In this section we recall the definition and some basic facts about surfaces isogenous to a higher product. For details, see {\cite{BC}}.

\begin{defi} A surface $S$ is called isogenous to a higher product if $S=(C \times D)/G$ where $C$, $D$ are curves with genus at least $2$ and $G$ is a finite group acting freely on $C \times D$. When $G$ acts via a product action, $S$ is called of unmixed type.
\end{defi}

\begin{rmk} \cite{BC} Let $S$ be a surface isogenous to a higher product of unmixed type. Then $S$ is a surface of general type. When $p_g=q=0$, elementary computations show that $K_S^2=8$, $C/G \cong D/G \cong \mathbb{P}^1$ and $|G|=(g_C-1)(g_D-1)$ where $g_C$ and $g_D$ denote the genus of $C$ and $D$, respectively.
\end{rmk}

When $G$ is an abelian group, Bauer and Catanese proved that there are $4$ types of surfaces isogenous to a higher product of unmixed type with $p_g=q=0$. Moreover they computed the dimensions of the families they form in \cite{BC}.

\begin{thm} \cite{BC} Let $S$ be a surface isogenous to a higher product $(C \times D) /G$ of unmixed type with $p_g=q=0$. If $G$ is abelian, then $G$ is one of the following groups : \\
(1) $(\mathbb{Z}/2)^3$, and these surfaces form an irreducible connected component of dimension 5 in their moduli space; \\
(2) $(\mathbb{Z}/2)^4$, and these surfaces form an irreducible connected component of dimension 4 in their moduli space; \\
(3) $(\mathbb{Z}/3)^2$, and these surfaces form an irreducible connected component of dimension 2 in their moduli space; \\
(4) $(\mathbb{Z}/5)^2$, and $S$ is the Beauville surface.
\end{thm}

In \cite{BC}, the authors also computed the first homology groups of these surfaces. Recently Shabalin \cite{Sh} corrected their computation and Bauer and Catanese corrected their mistake using Magma.

\begin{thm} \cite{BC2}, \cite{BCF}, \cite{Sh} Let $S$ be a surface isogenous to a higher product $(C \times D)/G$ of unmixed type with $p_g=q=0$ and assume $G$ to be abelian. Then we have the following : \\
(1) $H_1(S,\mathbb{Z}) \cong (\mathbb{Z}/2)^4 \oplus (\mathbb{Z}/4)^2$ for $G=(\mathbb{Z}/2)^3$; \\
(2) $H_1(S,\mathbb{Z}) \cong (\mathbb{Z}/4)^4$ for $G=(\mathbb{Z}/2)^4$; \\
(3) $H_1(S,\mathbb{Z}) \cong (\mathbb{Z}/3)^5$ for $G=(\mathbb{Z}/3)^2$; \\
(4) $H_1(S,\mathbb{Z}) \cong (\mathbb{Z}/5)^3$ for $G=(\mathbb{Z}/5)^2$.
\end{thm}

\begin{rmk}
Let $S$ be a surface with $p_g=q=0$ isogenous to a higher product $(C \times D)/G$ of unmixed type and let $G$ be abelian. From the exponential sequence \[ 0 \to \mathbb{Z} \to \mathcal{O} \to \mathcal{O}^* \to 0 \] we get \[ Pic(S) \cong H^2(S,\mathbb{Z}). \]

The above theorem and Noether's formula \[ \chi(\mathcal{O}_X) = 1 = \frac{1}{12}(8+2b_0-2b_1+b_2)=\frac{1}{12}(K_S^2 + \chi_{top}(S)) \] imply that these surfaces have $b_2=2$.

Finally the above theorem and the universal coefficient theorem imply the following : \\
(1) $Pic(S) \cong H^2(S,\mathbb{Z}) \cong \mathbb{Z}^2 \oplus (\mathbb{Z}/2)^4 \oplus (\mathbb{Z}/4)^2$ for $G=(\mathbb{Z}/2)^3$; \\
(2) $Pic(S) \cong H^2(S,\mathbb{Z}) \cong \mathbb{Z}^2 \oplus (\mathbb{Z}/4)^4$ for $G=(\mathbb{Z}/2)^4$; \\
(3) $Pic(S) \cong H^2(S,\mathbb{Z}) \cong \mathbb{Z}^2 \oplus (\mathbb{Z}/3)^5$ for $G=(\mathbb{Z}/3)^2$; \\
(4) $Pic(S) \cong H^2(S,\mathbb{Z}) \cong \mathbb{Z}^2 \oplus (\mathbb{Z}/5)^3$ for $G=(\mathbb{Z}/5)^2$.
\end{rmk}

Now we compute the Grothendieck groups of these surfaces. We will follow the arguments in {\cite[Proposition 2.1]{BBS}} and {\cite[Lemma 2.7]{GS}}.

\begin{lem} Let $S$ be a surface with $p_g=q=0$ isogenous to a higher product $(C \times D)/G$ of unmixed type and let $G$ be abelian. Then
\[ K(S) \cong \mathbb{Z}^2 \oplus Pic(S). \]
\end{lem}
\begin{proof}
We have the following standard isomorphisms
\[ F^0K(S)/F^1K(S) \cong CH^0(S) \cong \mathbb{Z}, \]
\[ F^1K(S)/F^2K(S) \cong Pic(S), \]
\[ F^2K(S) \cong CH^2(S). \]
From the result of Kimura {\cite{Ki}} (see also {\cite{BCP}}), we know that Bloch's conjecture holds for $S$, i.e. $ CH^2(S) \cong \mathbb{Z}.$ Therefore we get the following short exact sequence \[ 0 \to \mathbb{Z} \to F^1K(S) \to Pic(S) \to 0. \] Let $[p]$ be an element in $CH^2(S)$ represented by a point $p$ in $S$. Then $[p]$ is a basis of $CH^2(S) \cong \mathbb{Z}$, and $[\mathcal{O}_p]$ is the image of $[p]$ in $F^1K(S)$. By adjuntion, we have $\chi(\mathcal{O}_X,\mathcal{O}_p)=1$. This implies that the torsion free part of $[\mathcal{O}_p]$ in $F^1K(S)$ is an element of a basis of the torsion free part of $F^1K(S)$. It follows that the above short exact sequence splits and we get $F^1K(S) \cong \mathbb{Z} \oplus Pic(S)$. From the short exact sequence
\[ 0 \to F^1K(S) \to K(S) \to \mathbb{Z} \to 0 \] we get $K(S) \cong \mathbb{Z}^2 \oplus Pic(S).$
\end{proof}

\section{Derived categories of surfaces isogenous to a higher product with $G=(\mathbb{Z}/3)^2$}

In $\S$3 and $\S$4, we consider the derived categories of surfaces isogenous to a higher product of unmixed type with $p_g=q=0$, $G=(\mathbb{Z}/3)^2$. We recall some basic notions to describe the derived category of algebraic variety.

\begin{defi}
(1) An object $E$ in a triangulated category $D$ is called exceptional if
\begin{displaymath}
Hom(E,E[i])=\left \{ {\begin{array}{ll} \mathbb{C} & \textrm{if $i=0$,} \\ 0 & \textrm{otherwise.} \end{array}}
\right.
\end{displaymath}
(2) A sequence $E_1, \cdots, E_n$ of exceptional objects is called an exceptional sequence if $$ Hom(E_i,E_j[k])=0, \forall i > j, \forall k. $$
\end{defi}

When $S$ is a surface with $p_g=q=0$, every line bundle on $S$ is an exceptional object in $D^b(S)$. In this paper we want to prove the following theorem.

\begin{thm} Let $S=(C \times D)/G$ be a surface isogenous to a higher product of unmixed type with $p_g=q=0$, $G=(\mathbb{Z}/3)^2$. There are exceptional sequences of line bundles of maximal length on $S$. The orthogonal complements of the admissible subcategories in the derived category of $S$ are quasiphantom categories whose Grothendieck groups are isomorphic to $(\mathbb{Z}/3)^5$.
\end{thm}

We will construct exceptional sequences of line bundles of maximal length using $G$-equivariant line bundles on $C \times D$. For this we study the equivariant geometry of $C$ and $D$.

\subsection{Equivariant geometry of $C$}

From {\cite{BC}}, we see that $C$ is a curve with genus $4$. The group $G$ acts on $C$ and let $\pi : C \to \mathbb{P}^1$ be the quotient map. There are $4$ branch points on $\mathbb{P}^1$ and $4$ orbits on $C$ where the $G$ action has nontrivial stabilizers. Let $E_1, E_2, E_3, E_4$ be the set-theoretic orbits of ramification points.

Let $X$ be a smooth projective variety and let $G$ be a finite group acting on $X$. There is a well-known exact sequence
\[ 0 \to \widehat{G} \to Pic^G(X) \to Pic(X)^G \to H^2(G,\mathbb{C}^*), \]
and the last homomorphism is surjective when $X$ is a curve (see {\cite{D2}}).

When $G$ is abelian, Galkin and Shinder proved the following lemma.

\begin{lem}{\cite[Lemma 2.1]{GS}}
Let $G$ be a finite abelian group. Then the image of $Pic^G(X)$ in $Pic(X)^G$ consists of equivalence classes of $G$-invariant divisors and there is a short exact sequence \[ 0 \to \widehat{G} \to Pic^G(X) \to Div(X)^G/{\sim} \to 0, \]
where $ \sim $ denotes the linear equivalence.
\end{lem}

Using the above exact sequences, we analyze the equivariant geometry
of $C$.

\begin{nota}
From now to $\S$4, we let $G=(\mathbb{Z}/3)^2$ and $S=(C \times D)/G$ where $C$ and $D$ are curves with genus at least $2$ on which $G$ acts such that the diagonal action of $G$ on $C \times D$ is free.
\end{nota}

\begin{lem} (1) $Div(C)^G/\sim ~ \cong \mathbb{Z} \oplus \mathbb{Z}/3.$ \\
(2) There are exactly two $G$-invariant effective divisors of degree $3$ on $C$ which are not linearly equivalent.
\end{lem}
\begin{proof}
(1) follows from {\cite[Equation 2.2]{D2}} and the short exact sequence in the above lemma. From \cite{BC} we may
assume that the stabilizer elements of $E_1, E_2, E_3, E_4$ are
$e_1, e_2, -e_1, -e_2$, respectively, where $e_1, e_2$ are basis of
$G=(\mathbb{Z}/3)^2$. Consider $ \langle e_1 \rangle $-action on
$C$, and let $\phi : C \to \mathbb{P}^1$ be its quotient map. Then
we get $E_2 \sim E_4$ since each of them is a pullback of a point of
$\mathbb{P}^1$ via $\phi$. Similarly we get $E_1 \sim E_3$. If the
four orbits are all linearly equivalent then $Div(C)^G/\sim ~ \cong
\mathbb{Z}$ which contradicts to (1). Therefore the four orbits
cannot be all linearly equivalent. Since all $G$-invariant divisors
are linear combination of $G$-orbits, we get (2).
\end{proof}

\begin{lem}
\[ h^0(C, \mathcal{O}_C(2E_1-E_2)) = 0, \]
\[ h^1(C, \mathcal{O}_C(2E_1-E_2)) = 0, \]
\[ h^0(C, \mathcal{O}_C(E_2-2E_1)) = 0, \]
\[ h^1(C, \mathcal{O}_C(E_2-2E_1)) = 6. \]
\end{lem}
\begin{proof}
From the Riemann-Roch formula we find that $$ h^0(C, \mathcal{O}_C(2E_1-E_2)) - h^1(C, \mathcal{O}_C(2E_1-E_2)) = 3 +1-4 = 0. $$
Therefore it suffices to show that $h^0(C, \mathcal{O}_C(2E_1-E_2)) = 0.$ We know that $E_1, E_2$ are $G$-invariant divisors on $C$ and hence there is a $G$-action on $ H^0(C, \mathcal{O}_C(2E_1-E_2)). $ If $ h^0(C, \mathcal{O}_C(2E_1-E_2)) \neq 0, $ then there is a $G$-eigensection $f \in H^0(C, \mathcal{O}_C(2E_1-E_2)),$ and $2E_1-E_2+(f)$ should be a $G$-invariant effective divisor of degree $3$. Every $G$-invariant effective divisor of degree $3$ on $C$ is linearly equivalent to $E_1$ or $E_2$ by the above lemmas. It follows that $2E_1-E_2 \sim E_1$ or $2E_1-E_2 \sim E_2$. Then $ E_1-E_2 \sim 0 $ or $ 2E_1-2E_2 \sim 0 $ which contradicts the assumption that $E_1$ and $E_2$ are not linearly equivalent.

Similarly we get $$ h^0(C, \mathcal{O}_C(E_2-2E_1)) - h^1(C, \mathcal{O}_C(E_2-2E_1)) = -3 +1-4 = -6, $$
and $$ h^0(C, \mathcal{O}_C(E_2-2E_1)) = 0 $$ because the degree of $\mathcal{O}_C(E_2-2E_1)$ is negative.
\end{proof}

\begin{rmk}
From the same argument as above, we can find two set-theoretic orbits of ramification points $F_1$, $F_2$ on $D$ which are not linearly equivalent. Then we have
\[ h^0(D, \mathcal{O}_D(2F_1-F_2)) = 0, $$ $$ h^1(D, \mathcal{O}_D(2F_1-F_2)) = 0, \]
\[ h^0(D, \mathcal{O}_D(F_2-2F_1)) = 0, $$ $$ h^1(D, \mathcal{O}_D(F_2-2F_1)) = 6. \]
\end{rmk}

\subsection{Exceptional sequences of line bundles on $S$}

Let $X$ be the product of $C$ and $D$. By abuse of notation, we let $\mathcal{O}_X(E_i)$ (respectively, $\mathcal{O}_X(F_i)$) for $i \in \{1,2\}$ denote the pullback of $ \mathcal{O}_C(E_i) $ (respectively, $ \mathcal{O}_D(F_i)$). For any character $\chi \in Hom(G,\mathbb{C}^*)$, we can identify equivariant line bundles $\mathcal{O}_X(E_i)(\chi)$ (respectively, $ \mathcal{O}_X(F_i)(\chi) $) with line bundles on $S$.

\begin{thm} For any choice of $4$ characters $\chi_1, \chi_2, \chi_3, \chi_4$,
\[ \mathcal{O}_X(\chi_1), \mathcal{O}_X(E_2-2E_1)(\chi_2), \mathcal{O}_X(F_2-2F_1)(\chi_3), \mathcal{O}_X(E_2-2E_1+F_2-2F_1)(\chi_4) \] form an exceptional sequence of line bundles of maximal length on $S$.
\end{thm}
\begin{proof}
Since $p_g=q=0$, every line bundle on $S$ is exceptional.
From the K\"{u}nneth formula we find that
\[ h^j(X, \mathcal{O}_X(2E_1-E_2)) = 0, \forall j, \]
\[ h^j(X, \mathcal{O}_X(2F_1-F_2)) = 0, \forall j, \]
\[ h^j(X, \mathcal{O}_X(2E_1-E_2+2F_1-F_2)) = 0, \forall j, \]
\[ h^j(X, \mathcal{O}_X(-2E_1+E_2+2F_1-F_2)) = 0, \forall j. \]
Therefore the $G$-invariant parts are also trivial. Hence, we find that $\mathcal{O}_X(\chi_1)$, $\mathcal{O}_X(E_2-2E_1)(\chi_2)$, $\mathcal{O}_X(F_2-2F_1)(\chi_3)$, $\mathcal{O}_X(E_2-2E_1+F_2-2F_1)(\chi_4)$ form an exceptional sequence. Since $K(S) \cong \mathbb{Z}^4 \oplus (\mathbb{Z}/3)^5$, the maximal length of exceptional sequences on $S$ is $4$.
\end{proof}

\subsection{Deformations of categories generated by exceptional sequences}

In this subsection we discuss about the deformations of categories generated by exceptional sequences. In order to do this we recall definitions and basic facts about $A_{\infty}$-algebras. For details, see {\cite{Ke}}.

\begin{defi}{\cite{Ke}}
An $A_{\infty}$-algebra is a $\mathbb{Z}$-graded vector space
\[ A = \bigoplus_{p \in \mathbb{Z}} A^p \]
endowed with graded maps
\[ m_n : A^{\otimes n} \to A, n \geq 1, \]
of degree $2-n$ satisfying
\[ \sum (-1)^{r+st}m_{r+1+t}(1^{\otimes r}\otimes m_s \otimes 1^{\otimes t}) = 0, \]
where the sum runs over all decompositions $n=r+s+t$.
\end{defi}

\begin{defi}{\cite{Ke}}
An $A_{\infty}$-algebra $A$ is called strictly unital if it has an element $1$ of degree zero such that $m_1(1)=0$, $m_2(1,a)=m_2(a,1)=a$ for all $a \in A$ and for $n \geq 3$, $m_n(a_1,\cdots,a_n)=0$ if one of $a_i \in \{a_1, \cdots a_n\} \subset A$ is equal to $1$.
\end{defi}

We want to prove that the $A_{\infty}$-algebra of endomorphism of the exceptional sequences constructed above is formal. We follows the arguments in {\cite{AO}}, {\cite{BBS}} and {\cite{GS}}.

\begin{prop}
$T = \mathcal{O}_X(\chi_1) \oplus \mathcal{O}_X(E_2-2E_1)(\chi_2) \oplus \mathcal{O}_X(F_2-2F_1)(\chi_3) \oplus \mathcal{O}_X(E_2-2E_1+F_2-2F_1)(\chi_4)$, and let $B=RHom(T,T)$ be the DG-algebra of endomorphisms. Then $B$ is formal, i.e. $H^*(B)$ can be chosen to be a graded algebra.
\end{prop}
\begin{proof}
Consider the minimal model $H^*(B)$ of $B$. We want to show that $H^*(B)$ does not have nontrivial $m_n$ for $n \geq 3$. From the K\"{u}nneth formula we have the following :
$$ H^k(X,\mathcal{O}_X(2E_1-E_2-2F_1+F_2))=0, ~~ \forall k \in \mathbb{Z}. $$
$$ H^k(X,\mathcal{O}_X(2F_1-F_2-2E_1+E_2))=0, ~~ \forall k \in \mathbb{Z}. $$

By {\cite[Lemma 2.1]{Sei}}, we may assume that the $H^*(B)$ is strictly unital. Consider arbitrary $m_n(b_1,\cdots,b_n)$ for $n \geq 3$. The above computations show that if every $b_i$ is nonzero then at least one $b_i$ should be multiple of $1$. Therefore $m_n(b_1,\cdots,b_n)=0$, for all $n \geq 3$.
\end{proof}

In \cite{AO}, Alexeev and Orlov asked the following question.

\begin{question} \cite{AO}
Is is true that for any exceptional collection of maximal length on a smooth projective surface $S$ with ample $K_S$ and with $p_g=q=0$, the DG algebra of endomorphisms of the exceptional collection does not change under small deformations of the complex structure on $X$?
\end{question}

They constructed exceptional sequences of maximal length for primary
Burniat surfaces and proved that the above property holds for their
exceptional sequences in \cite{AO}. This phenomenon was observed for
other general type surfaces ( see \cite{BBKS}, \cite{BBS}, \cite{C}
). We prove that for certain exceptional sequence we constructed the
above question is true.

\begin{prop}
There is a choice of four characters $\chi_i \in \widehat{G}, i \in \{1,2,3,4\}$ such that the DG algebra of endomorphisms of $T = \mathcal{O}_X(\chi_1) \oplus \mathcal{O}_X(E_2-2E_1)(\chi_2) \oplus \mathcal{O}_X(F_2-2F_1)(\chi_3) \oplus \mathcal{O}_X(E_2-2E_1+F_2-2F_1)(\chi_4)$ does not change under small deformations of the complex structure of $S$.
\end{prop}

\begin{proof}

From the Riemann-Roch theorem for curves and K\"{u}nneth formula we have the following :
\begin{displaymath}
H^k(X,\mathcal{O}_X(E_2-2E_1))=\left \{ {\begin{array}{ll} \mathbb{C}^6 & \textrm{if $k=1$,} \\ \mathbb{C}^{24} & \textrm{if $k=2$,} \\ 0 & \textrm{otherwise.} \end{array}}
\right.
\end{displaymath}

\begin{displaymath}
H^k(X,\mathcal{O}_X(F_2-2F_1))=\left \{ {\begin{array}{ll} \mathbb{C}^6 & \textrm{if $k=1$,} \\ \mathbb{C}^{24} & \textrm{if $k=2$,} \\ 0 & \textrm{otherwise.} \end{array}}
\right.
\end{displaymath}

\begin{displaymath}
H^k(X,\mathcal{O}_X(E_2-2E_1+F_2-2F_1))=\left \{ {\begin{array}{ll} \mathbb{C}^{36} & \textrm{if $k=2$,} \\ 0 & \textrm{otherwise.} \end{array}}
\right.
\end{displaymath}

Then there is a $\chi,\chi' \in \widehat{G}$ such that $H^1(X,\mathcal{O}_X(E_2-2E_1)(\chi))^G=0$ and $H^1(X,\mathcal{O}_X(F_2-2F_1)(\chi'))^G=0$.

From the Riemann-Roch theorem for surfaces we get for any character $\chi+\chi' \in \widehat{G}$,
\begin{displaymath}
H^k(X,\mathcal{O}_X(E_2-2E_1+F_2-2F_1)(\chi+\chi'))^G=\left \{ {\begin{array}{ll} \mathbb{C}^{4} & \textrm{if $k=2$,} \\ 0 & \textrm{otherwise.} \end{array}}
\right.
\end{displaymath}

Therefore there is a choice of four characters $\chi_i \in \widehat{G}, i \in \{1,2,3,4\}$ such that the minimal model of the DG algebra of endomorphisms of $T = \mathcal{O}_X(\chi_1) \oplus \mathcal{O}_X(E_2-2E_1)(\chi_2) \oplus \mathcal{O}_X(F_2-2F_1)(\chi_3) \oplus \mathcal{O}_X(E_2-2E_1+F_2-2F_1)(\chi_4)$ has only terms in degree 0 and 2. The multiplication of two elements of degree 2 is 0 since there is no $Ext^4$ between objects. Hence the structure of the DG-algebra is completely determined in this case.
\end{proof}

We do not know whether the DG algebra of endomorphism of $T = \mathcal{O}_X(\chi_1) \oplus \mathcal{O}_X(E_2-2E_1)(\chi_2) \oplus \mathcal{O}_X(F_2-2F_1)(\chi_3) \oplus \mathcal{O}_X(E_2-2E_1+F_2-2F_1)(\chi_4)$ does not change under small deformations of complex structure of $S$ for every choice of four characters $\chi_i \in \widehat{G}, i \in \{1,2,3,4\}$.

\section{Quasiphantom categories and phantom categories}

In this section we consider Hochschild homologies and cohomologies of the orthogonal complements of the categories generated by exceptional sequences.

\subsection{Hochschild homology and cohomology}
We recall the definition and some basic facts about Hochschild homology and cohomology of a smooth projective variety. For details about Hochchild homology and cohomology, see {\cite{Ku}}.

\begin{defi}{\cite{Ku}}
Let $S$ be a smooth projective variety. The Hochschild  homology and cohomology of $S$ are defined by
\[ HH_*(S) = Hom^*(S \times S, \Delta_*\mathcal{O}_S \otimes \Delta_*\mathcal{O}_S), \]
\[ HH^*(S) = Hom^*_{S \times S}(\Delta_*\mathcal{O}_S, \Delta_*\mathcal{O}_S). \]
\end{defi}

Hochschild homology and cohomology of a smooth projective variety can be computed using the following theorem.

\begin{thm}(Hochschild-Kostant-Rosenberg isomorphisms){\cite[Theorem 8.3]{Ku}}Let $S$ be a smooth projective variety of dimension $n$. Then
\[ HH_t(S) \cong \bigoplus_{p=0}^{n} H^{t+p}(S, \Omega_S^p), \]
\[ HH^t(S) \cong \bigoplus_{p=0}^{n} H^{t-p}(S, \wedge ^p \mathrm{T}_S). \]
\end{thm}

Let $S$ be a smooth projective variety. Kuznetsov furthermore defined Hochshild homology and cohomology for any admissible subcategory $\mathcal{A} \subset D^b(S)$ in {\cite{Ku}}.

\begin{defi}{\cite[Definition 4.4]{Ku}}
Let $S$ be a smooth projective variety, and $\mathcal{A} \subset D^b(S)$ be an admissible subcategory. Let $\mathcal{E}_\mathcal{A}$ be a strong generator of $\mathcal{A}$ and $\mathcal{C}_{\mathcal{A}}=RHom^*(\mathcal{E}_\mathcal{A},\mathcal{E}_\mathcal{A})$. Then the Hochschild homology and cohomology of $\mathcal{A}$ are defined as follows :
\[ HH_*(\mathcal{A}) := \mathcal{C}_{\mathcal{A}} \otimes^L_{\mathcal{C}_{\mathcal{A}} \otimes \mathcal{C}_{\mathcal{A}}^{opp}} \mathcal{C}_{\mathcal{A}}, \]
\[ HH^*(\mathcal{A}) := RHom_{\mathcal{C}_{\mathcal{A}} \otimes \mathcal{C}_{\mathcal{A}}^{opp}}(\mathcal{C}_{\mathcal{A}},\mathcal{C}_{\mathcal{A}}). \]
\end{defi}

Then Kuznetsov proved the additivity of Hochschild homology with respect to the semiorthogonal decomposition in {\cite{Ku}} which is a main tool to compute the Hochschild homology of the orthogonal complement of an admissible subcategory.

\begin{thm}{\cite[Corollary 7.5, Corollary 8.4]{Ku}}
(1) For any semiorthogonal decomposition $D^b(S) = \langle \mathcal{A}_1, \cdots, \mathcal{A}_n \rangle$, there is an isomorphism \[ HH_*(S) \cong HH_*(\mathcal{A}_1) \oplus \cdots \oplus HH_*(\mathcal{A}_n). \]
(2) If $E$ is an exceptional object in $D^b(S)$, then $HH_*(\langle E \rangle) \cong HH^*(\langle E \rangle) \cong \mathbb{C}$.
\end{thm}

In the rest of this subsection we will follow \cite{Ku2} to compute the Hochschild cohomologies of the orthogonal complements of the exceptional sequences. Let $S$ be a smooth projective variety and let $\mathcal{D}$ be the $\check{C}ech$ enhancement of $D^b(S)$ and $\mathcal{E} \subset \mathcal{D}$ a DG-subcategory.

\begin{defi} \cite{Ku2} The normal Hochschild cohomology of $\mathcal{E}$ of $\mathcal{D}$ is defined as follows
$$ NHH^*(\mathcal{E},\mathcal{D}):=\mathcal{E} \otimes_{\mathcal{E}^{opp} \otimes \mathcal{E}} \mathcal{D}_{\mathcal{E}}^{\vee}. $$
\end{defi}

Then Kuznetsov proved that the normal Hochschild cohomology can be used to study the restriction of Hochschild cohomology of $S$ to Hochschild cohomology of orthogonal complement of exceptional sequence $E_1, \cdots, E_n$ in the following way.

\begin{thm} \cite{Ku2} If $ D^b(S) = \langle \mathcal{A}, E_1, \cdots, E_n \rangle $ is a semiorthogonal decomposition then there is distinguished triangle
$$ NHH^*(\mathcal{E},\mathcal{D}) \to HH^*(\mathcal{D}) \to HH^*(\mathcal{A}), $$
where $\mathcal{E}$ is the DG-category generated by $E_1, \cdots, E_n$.
\end{thm}

Then Kuznetsov defined the notion of height.

\begin{defi} \cite{Ku2} The height of an exceptional sequence $E_1, \cdots, E_n$ is defined as
$$ h(E_1, \cdots, E_n)=min\{k \in \mathbb{Z} | NHH^k(\mathcal{E},\mathcal{D}) \neq 0 \}, $$
where $\mathcal{E}$ is the DG-category generated by $E_1, \cdots, E_n$.
\end{defi}

Then the above distinguished triangle gives the following theorem.

\begin{thm} \cite{Ku2} Let $h=h(E_1,\cdots,E_n)$ be the height of an exceptional sequence $E_1, \cdots, E_n$ and let $\mathcal{A}$ be its orthogonal complement. The canonical restriction morphism $HH^k(X) \to HH^k(\mathcal{A})$ is an isomorphism for $k \leq h-2$ and a monomorphism for $k=h-1$.
\end{thm}

In order to compute height Kuznetsov introduced the notion of pseudoheight.

\begin{defi} \cite{Ku2} (1) For any two object $E,E'$ we define their relative height as $e(E,E')=min \{ k|Ext^k(E,E') \neq 0 \}.$ \\
(2) The pseudoheight $ph(E_1, \cdots, E_n)$ of the exceptional sequence $E_1, \cdots, E_n$ is
$$ ph(E_1, \cdots, E_n)=min_{1 \leq a_0 < a_1 < \cdots < a_p \leq n} (e(E_{a_0},E_{a_1})+\cdots+e(E_{a_{p-1}},E_{a_p})+e(E_{a_p},S^{-1}(E_{a_0}))-p). $$
where $S$ is the Serre functor.
\end{defi}

Kuznetsov proved that there is a spectral sequence which converges to the normal Hochschild cohomology and pseudoheight is the minimum of total degrees of nonzero terms of the 1st page of this spectral sequence. Therefore he got the following inequality.

\begin{lem} \cite{Ku2}
$$ h(E_1, \cdots, E_n) \geq ph(E_1, \cdots, E_n). $$
\end{lem}

Finally he proved that for some cases it is easy to determine the pseudoheight.

\begin{defi} \cite{Ku2} Let $E_1, \cdots E_n$ be an exceptional sequence on $S$. \\
(1) $E_1, \cdots E_n,E_1 \otimes \omega_S^{-1}, \cdots E_n \otimes \omega_S^{-1}$ is called anticanonically extended collection. \\
(2) $E_1, \cdots E_n,E_1 \otimes \omega_S^{-1}, \cdots E_n \otimes \omega_S^{-1}$ is called Hom-free if $Ext^k(E_i,E_j)$ for $k \leq 0$ and all $i < j \leq i+n$. \\
(3) A Hom-free anticanonically extended sequence is cyclically $Ext^1$-connected if there is a chain $1 \leq a_0 < a_1 < \cdots < a_p \leq n$ such that $Ext^1(E_{a_s},E_{a_{s+1}}) \neq 0$, for $s=0,1,\cdots, p-1$ and $Ext^1(E_{a_p},E_{a_0} \otimes \omega^{-1}) \neq 0.$
\end{defi}

When $E_1, \cdots E_n,E_1 \otimes \omega_S^{-1}, \cdots E_n \otimes \omega_S^{-1}$ is Hom-free and not cyclically $Ext^1$-connected the following two lemmas enable us to compute the height of the exceptional sequence.

\begin{lem}\cite{Ku2} If $E_1, \cdots E_n$ is an exceptional sequence such that its anticanonically extended sequence is Hom-free then $ph(E_1, \cdots E_n) \geq 1 + dim S$. If this sequence is not cyclically $Ext^1$-connected then $ph(E_1, \cdots E_n) \geq 2 + dim S$.
\end{lem}

\begin{lem}\cite{Ku2}
Let $S$ be a smooth projective surface with $H^2(S,\omega_S^{-1}) \neq 0$. If $E_1, \cdots E_n$ is an exceptional sequence of line bundles then $ph(E_1, \cdots E_n) \leq 4$. Moreover if $ph(E_1, \cdots E_n) = 4$ then $h(E_1, \cdots E_n) = 4$.
\end{lem}

Now we compute the heights of our exceptional sequences.

\begin{prop}
The pseudoheight of the exceptional collection $ \mathcal{O}_X(\chi_1), \mathcal{O}_X(E_2-2E_1)(\chi_2),$ $\mathcal{O}_X(F_2-2F_1)(\chi_3), \mathcal{O}_X(E_2-2E_1+F_2-2F_1)(\chi_4) $ is 4 and the height is 4.
\end{prop}
\begin{proof}
It is enough to show that the exceptional sequence is Hom-free and not cyclically $Ext^1$-connected. From the K\"{u}nneth formula and degree computation we find that $ \mathcal{O}_X(\chi_1), \mathcal{O}_X(E_2-2E_1)(\chi_2), \mathcal{O}_X(F_2-2F_1)(\chi_3), \mathcal{O}_X(E_2-2E_1+F_2-2F_1)(\chi_4), \mathcal{O}_X(\chi_1) \otimes \omega_S^{-1}, \mathcal{O}_X(E_2-2E_1)(\chi_2) \otimes \omega_S^{-1}, \mathcal{O}_X(F_2-2F_1)(\chi_3) \otimes \omega_S^{-1}, \mathcal{O}_X(E_2-2E_1+F_2-2F_1)(\chi_4) \otimes \omega_S^{-1} $ is Hom-free. This sequence cannot be cyclically $Ext^1$-connected by Serre duality and Kodaira vanishing theorem.
\end{proof}

Therefore we get the following consequence about the Hochschild cohomologies of the orthogonal complements of our exceptional sequences.

\begin{cor}
Let $\mathcal{A}$ be the orthogonal complement of the exceptional collection $ \mathcal{O}_X(\chi_1),$ $\mathcal{O}_X(E_2-2E_1)(\chi_2), \mathcal{O}_X(F_2-2F_1)(\chi_3), \mathcal{O}_X(E_2-2E_1+F_2-2F_1)(\chi_4) $. Then we have $ HH^i(S) = HH^i(\mathcal{A})$, for $i=0,1,2 $, and $ HH^3(S) \subset HH^3(\mathcal{A}) $.
\end{cor}

\subsection{Quasiphantom categories and phantom categories}

We recall the definitions of quasiphantom and phantom category.

\begin{defi}{\cite[Definition 1.8]{GO}}
Let $S$ be a smooth projective variety. Let $\mathcal{A}$ be an admissible triangulated subcategory of $D^b(S)$. Then $\mathcal{A}$ is called a quasiphantom category if the Hochschild homology of $\mathcal{A}$ vanishes, and the Grothendieck group of $\mathcal{A}$ is finite. If the Grothendieck group of $\mathcal{A}$ also vanishes, then $\mathcal{A}$ is called a phantom category.
\end{defi}

We now prove the second part of our main theorem.

\begin{prop}
Let $\mathcal{A}$ be the left orthogonal complement of the admissible category generated by $\mathcal{O}_X(\chi_1)$, $\mathcal{O}_X(E_2-2E_1)(\chi_2)$, $\mathcal{O}_X(F_2-2F_1)(\chi_3)$, $\mathcal{O}_X(E_2-2E_1+F_2-2F_1)(\chi_4)$. Then we have $K(\mathcal{A}) = (\mathbb{Z}/3)^5$ and $ HH_*(\mathcal{A}) = 0$. Therefore $\mathcal{A}$ is a quasiphantom category.
\end{prop}
\begin{proof}
Additivity of Grothendieck groups with respect to semiorthogonal decomposition implies that $ K(\mathcal{A}) = (\mathbb{Z}/3)^5.$ From the Hochschild-Kostant-Rosenberg isomorphism the total dimension of Hochschild homology of $S$ is the sum of Betti number of $S$ which is $4$. From our construction and Kuznetsov's theorem, we get the total dimension of the Hochschild homology of the admissible subcategory generated by exceptional sequence above is also $4$. Vanishing of Hochschild homology of $\mathcal{A}$ follows directly from Kuznetsov's theorem.
\end{proof}

Gorchinskiy and Orlov proved the following theorem and constructed phantom categories using the quasiphantom categories constructed in {\cite{AO}}, {\cite{BBS}}, {\cite{GS}}.

\begin{thm}{\cite[Theorem 1.12]{GO}}
Let $S$, $S'$ be smooth surfaces with $p_g=q=0$ for which Bloch's conjecture holds. Assume that the derived categories $D^b(S)$ and $D^b(S')$ have exceptional collections of maximal lengths. Let $\mathcal{A} \subset D^b(S)$ and $\mathcal{A'} \subset D^b(S')$ be the left orthogonals to these exceptional collections. If the orders of $Pic(S)_{tors}$ and $Pic(S')_{tors}$ are coprime, then the admissible subcategory $\mathcal{A} \boxtimes \mathcal{A}' \subset D^b(S \times S')$ is a phantom category.
\end{thm}

The classical Godeaux surface, primary Burniat surfaces and the Beuaville surface are surfaces with $p_g=q=0$ satisfying Bloch's conjecture. Their Picard groups are $\mathbb{Z}^{11} \oplus (\mathbb{Z}/5)$, $\mathbb{Z}^6 \oplus (\mathbb{Z}/2)^6$ and $\mathbb{Z}^4 \oplus (\mathbb{Z}/5)^3$, respectively. Finally they have exceptional sequences of line bundles of maximal lengths({\cite{AO}}, {\cite{BBS}}, {\cite{GS}}). Hence we get the following proposition.

\begin{prop}
Quasiphantom categories of surfaces isogenous to a higher product with $p_g=q=0$, $G=(\mathbb{Z}/3)^2$ and quasiphantom categories constructed in {\cite{AO}}, {\cite{BBS}}, {\cite{GS}} make phantom categories by the above theorem of Gorchinskiy and Orlov.
\end{prop}

\section{Discussions}

The construction of exceptional sequences of line bundles on $S=(C \times D)/G$ of this paper does not extend to the cases where $G=(\mathbb{Z}/2)^3$, $(\mathbb{Z}/2)^4$, $(\mathbb{Z}/5)^2$. For the $G=(\mathbb{Z}/5)^2$ case, $Div(C)^G/{\sim} \cong \mathbb{Z}$ {\cite[Lemma 2.3]{GS}}. (However there are exceptional sequences of line bundles of maximal length due to the construction of Galkin and Shinder.) For the $G=(\mathbb{Z}/2)^3$, $G=(\mathbb{Z}/2)^4$ cases, the following propositions imply that if there is an exceptional sequence of line bundles on $S$ we need another construction to find it.

\begin{prop}
Let $G=(\mathbb{Z}/2)^3$, $X:=C \times D$ and $S:=(C \times D)/G$ be a surface isogenous to a higher product with $p_g=q=0$ of unmixed type. Then $C$ is a curve of genus $3$ and $D$ is a curve of genus $5$ (see {\cite{BC}}). Let $E_1, E_2$ be $2$ linear combinations of set-theoretic orbits on $C$, and $F_1, F_2$ be $2$ linear combinations of set-theoretic orbits on $D$. By abuse of notation, we let $\mathcal{O}_X(E_i)$ (respectively, $\mathcal{O}_X(F_i)$) denote the pullback of $ \mathcal{O}_C(E_i) $ (respectively, $ \mathcal{O}_D(F_i)$) for $i \in \{1,2\}$. For any choice of characters $\chi_1, \chi_2 \in Hom(G,\mathbb{C}^*)$, we can identify the equivariant line bundles $\mathcal{O}_X(E_i+F_i)(\chi_i)$ on $X$ with line bundles on $S$ for $i \in \{1,2\}$. Then $$ \mathcal{O}_X,\mathcal{O}_X(E_1+F_1)(\chi_1),\mathcal{O}_X(E_2+F_2)(\chi_2) $$ cannot be an exceptional sequence on $S$.
\end{prop}
\begin{proof}
Suppose that $ \mathcal{O}_X,\mathcal{O}_X(E_1+F_1)(\chi_1),\mathcal{O}_X(E_2+F_2)(\chi_2) $ is an exceptional sequence of line bundles on $S$. Then we get $\chi(\mathcal{O}_X(E_i+F_i)(\chi_i),\mathcal{O}_X)=0$ for $i \in \{1,2\}$, and $\chi(\mathcal{O}_X(E_2+F_2)(\chi_2),\mathcal{O}_X(E_1+F_1)(\chi_1))=0$. From Riemann-Roch formula we get
\[ \chi(\mathcal{O}_X(-E_i-F_i)(\chi_i^{-1})) = \chi(\mathcal{O}_S)+\frac{1}{16} \mathcal{O}(-E_i-F_i) \cdot \mathcal{O}(-E_i-F_i+K_C+K_D) = \frac{1}{8}(e_i-2)(f_i-4) =0 , \]
where $e_i=$degree of $E_i$, $f_i=$degree of $F_i$, $i \in \{1,2\}$.
Degrees of all set-theoretic orbits on $C$ and $D$ are multiples of $4$. Therefore we get $f_1=f_2=4$. However then we have
\[ \chi(\mathcal{O}_X(E_1+F_1-E_2-F_2)(\chi_1+\chi_2^{-1})) = \chi(\mathcal{O}_S)+\frac{1}{16} \mathcal{O}(E_1+F_1-E_2-F_2) \cdot \mathcal{O}(E_1+F_1-E_2-F_2+K_C+K_D) \]
\[ = \frac{1}{8}(e_2-e_1-2)(f_2-f_1-4) \neq 0. \]
Therefore $ \mathcal{O}_X,\mathcal{O}_X(E_1+F_1)(\chi_1),\mathcal{O}_X(E_2+F_2)(\chi_2) $ cannot be an exceptional sequence on $S$.
\end{proof}

\begin{prop}
Let $G=(\mathbb{Z}/2)^4$, $X:=C \times D$ and $S:=(C \times D)/G$ be a surface isogenous to a higher product with $p_g=q=0$ of unmixed type. Then $C$ and $D$ are curves of genus $5$ (see {\cite{BC}}). Let $E$ be a linear combination of set-theoretic orbits on $C$, $F$ be a linear combination of set-theoretic orbits on $D$. By abuse of notation, we let $\mathcal{O}_X(E)$ (respectively, $\mathcal{O}_X(F)$) denote the pullback of $ \mathcal{O}_C(E) $ (respectively, $ \mathcal{O}_D(F)$). For any character $\chi \in Hom(G,\mathbb{C}^*)$, we can identify the equivariant line bundles $\mathcal{O}_X(E+F)(\chi)$ on $X$ with a line bundles on $S$. Then $$ \mathcal{O}_X,\mathcal{O}_X(E+F)(\chi) $$ cannot be an exceptional sequence on $S$.
\end{prop}
\begin{proof}
If $ \mathcal{O}_X,\mathcal{O}_X(E+F)(\chi) $ is an exceptional sequence of line bundles, then $\chi(\mathcal{O}_X(E+F)(\chi),\mathcal{O}_X)=0$. From Riemann-Roch formula we get
\[ \chi(\mathcal{O}_X(-E-F)(\chi^{-1})) = \chi(\mathcal{O}_S)+\frac{1}{32} \mathcal{O}(-E-F) \cdot \mathcal{O}(-E-F+K_C+K_D) = \frac{1}{16}(e-4)(f-4), \]
where $e=$degree of $E$, $f=$degree of $F$.
However degrees of all set-theoretic orbits on $C$ and $D$ are multiples of $8$. Therefore $\chi(\mathcal{O}_X(E+F)(\chi),\mathcal{O}_X) \neq 0$ and $ \mathcal{O}_X,\mathcal{O}_X(E+F)(\chi) $ cannot be an exceptional sequence on $S$.
\end{proof}

In the forthcoming paper \cite{Lee} we will show that there exist exceptional sequences of line bundles of maximal length on surfaces isogenous to a higher product of unmixed type with $p_g=q=0$, $G=(\mathbb{Z}/2)^3$ or $G=(\mathbb{Z}/2)^4$ via different method. The structures of quasiphantom categories are still mysterious and interesting.

\end{document}